\documentclass[11pt,leqno]{amsart}
\usepackage{verbatim,amssymb,amsfonts,latexsym,amsmath,amsthm,bbm,nicefrac,xspace,enumerate}

\newtheorem{theorem}{Theorem}[section]
\newtheorem{lemma}[theorem]{Lemma}
\newtheorem{corollary}[theorem]{Corollary}

\newtheorem{question}[theorem]{Question}
\newtheorem*{claim}{Claim}
\theoremstyle{definition}

\newtheorem*{observation}{Observation}
\theoremstyle{remark}

\newcommand{\B}{\mathcal{B}}
\newcommand{\C}{\mathcal{C}}
\newcommand{\D}{\mathcal{D}}
\newcommand{\F}{\mathcal{F}}
\newcommand{\G}{\mathcal{G}}

\renewcommand{\P}{\mathcal{P}}

\newcommand{\explicitSet}[1]{\left\lbrace #1 \right\rbrace}
\newcommand{\brackets}[1]{\left\langle #1 \right\rangle}
\newcommand{\set}[2]{\explicitSet{#1 \colon #2}}
\newcommand{\seq}[2]{\brackets{#1 \colon #2}}
\newcommand{\<}{\langle}
\renewcommand{\>}{\rangle}
\renewcommand{\a}{\alpha}
\renewcommand{\b}{\beta}
\newcommand{\g}{\gamma}
\newcommand{\dlt}{\delta}
\newcommand{\e}{\varepsilon}

\renewcommand{\k}{\kappa}
\newcommand{\s}{\sigma}

\newcommand{\w}{\omega}
\newcommand{\0}{\emptyset}

\newcommand{\sub}{\subseteq}
\newcommand{\rest}{\!\restriction\!}
\newcommand{\cat}{\!\,^{\frown}}

\newcommand{\homeo}{\approx}

\newcommand{\closure}[1]{\overline{#1}}

\newcommand{\cf}{\mathrm{cf}}
\newcommand{\card}[1]{\left\lvert #1 \right\rvert}
\newcommand{\tr}[1]{[\![#1]\!]}
\newcommand{\PP}{\mathbb{P}}
\newcommand{\forces}{\Vdash}

\newcommand{\continuum}{\mathfrak{c}}

\newcommand{\cov}[1]{\ensuremath{\bld{cov}(\scr{#1})}}

\newcommand{\gch}{\ensuremath{\mathsf{GCH}}\xspace}
\newcommand{\zfc}{\ensuremath{\mathsf{ZFC}}\xspace}

\newcommand{\ma}{\ensuremath{\mathsf{MA}}\xspace}

\newcommand{\wt}{wt}
\renewcommand{\mp}{\mathfrak{par}\hspace{-.1mm}}
\renewcommand{\cov}{\mathfrak{cov}\hspace{-.15mm}}
\newcommand{\cofinal}{\mathrm{cf} \hspace{-.1mm} \big( [\kappa]^{\aleph_0},\sub \hspace{-1mm} \big)}
\newcommand{\chang}{$(\aleph_{\omega+1},\aleph_\omega) \hspace{-.5mm} \twoheadrightarrow \hspace{-.5mm} (\aleph_1,\aleph_0)$\xspace}
\newcommand{\schh}{\ensuremath{\mathsf{SCH}^+}\xspace}

\begin{document}

\title[Covering versus partitioning with Polish spaces]{Covering versus partitioning \linebreak with Polish spaces}
\author{Will Brian}
\address {
W. R. Brian\\
Department of Mathematics and Statistics\\
University of North Carolina at Charlotte\\
9201 University City Blvd.\\
Charlotte, NC 28223, USA}
\email{wbrian.math@gmail.com}
\urladdr{wrbrian.wordpress.com}



\begin{abstract}

\vspace{-1.5mm}

Given a completely metrizable space $X$, let $\mp(X)$ denote the smallest possible size of a partition of $X$ into Polish spaces, and $\cov(X)$ the smallest possible size of a covering of $X$ with Polish spaces.
Observe that $\cov(X) \leq \mp(X)$ for every $X$, because every partition of $X$ is also a covering. 

We prove it is consistent relative to a huge cardinal that the strict inequality $\cov(X) < \mp(X)$ can hold for some completely metrizable space $X$.
We also prove that using large cardinals is necessary for obtaining this strict inequality, because if $\cov(X) < \mp(X)$ for any completely metrizable $X$, then $0^\dagger$ exists.
\end{abstract}

\maketitle


\vspace{-3mm}

\section{Introduction}

In this paper we examine coverings and partitions of large completely metrizable spaces with small ones. 
By a ``small'' completely metrizable space we mean one with a countable basis, or (equivalently) a countable dense subset. This is the familiar class of Polish spaces.
A ``large'' completely metrizable space means anything else.

For each completely metrizable space $X$, define
\begin{align*}
\cov(X) &=\, \min\set{\card{\C}}{\C \text{ is a covering of } X \text{ with Polish spaces}}\!, \text{ } \\
\mp(X) &=\, \min\set{\card{\P}}{\P \text{ is a partition of } X \text{ into Polish spaces}}\!.
\end{align*}
Observe that $\cov(X)$ and $\mp(X)$ are well defined and both $\leq\!|X|$, because we may cover/partition $X$ with singletons. Also, $\cov(X) \leq \mp(X)$ for every $X$, because every partition of $X$ is also a covering of $X$.
We are interested in two basic questions: (1) Can we determine $\cov(X)$ or $\mp(X)$ in terms of more familiar invariants? (2) Is it possible to have $\cov(X) < \mp(X)$?

Another preliminary observation is that $\cov(X) = \mp(X) = |X|$ whenever $|X| > \continuum$.
This is because every Polish space has cardinality $\leq\!\continuum$, and therefore a collection $\F$ of Polish spaces covers at most $|\F| \cdot \continuum$ points.

This gives an easy answer to both of our questions above when $|X| > \continuum$. Consequently, we are mostly interested in covering and partitioning non-Polish spaces with cardinality $\leq\!\continuum$.
If $X$ is such a space, then
$\aleph_1 \leq \cov(X)$ (otherwise $X$ would be separable) and $\mp(X) \leq \continuum$ (because we may partition $X$ into singletons).
In this way, $\cov(X)$ and $\mp(X)$ bear some resemblance to cardinal characteristics of the continuum.

Recall that the \emph{weight} of a topological space $X$, denoted $\wt(X)$, is the least cardinality of a basis for $X$.
The main results of this paper are:
\begin{itemize}
\item[$\circ$] If $X$ is any completely metrizable space with $\aleph_1 \leq \wt(X) < \aleph_\w$, then $\cov(X) = \mp(X) = \wt(X)$.

\vspace{1mm}

\item[$\circ$] It is consistent relative to a huge cardinal that $\cov(X) < \mp(X)$ for a broad class of completely metrizable spaces $X$ with $\aleph_\w \leq \wt(X) < \continuum$. This can be achieved, for example, by adding $>\! \aleph_{\w+1}$ Cohen reals to a model of $\gch + {}\!\!$ the Chang Conjecture \chang. 

\vspace{1mm}

\item[$\circ$] Suppose that for every singular cardinal $\mu < \continuum$ with $\cf(\mu) = \w$, a weak form of $\square_\mu$ holds (namely $\square_{\w_1,\mu}^{***}$), and $\mathrm{cf} \hspace{-.1mm} \big( [\mu]^{\aleph_0},\sub \hspace{-1mm} \big) = \mu^+$.
Then $\cov(X) = \mp(X)$ for every completely metrizable space $X$, 
and the value of $\cov(X)$ and $\mp(X)$ is determined by basic topological properties of $X$ (see Theorem~\ref{thm:main}).
Furthermore, the negation of these hypotheses implies $0^\dagger$ exists; consequently, if $\cov(X) < \mp(X)$ for any completely metrizable space $X$, then $0^\dagger$ exists.
\end{itemize}
These results are proved in Sections 2, 3, and 4, respectively.

The second and third bullet points above answer a question raised several years ago by the author and Arnie Miller \cite[Question 3.8]{Brian&Miller}. At least, the question is answered to the satisfaction of the first author of \cite{Brian&Miller}.


\section{$\cov$ and $\mp$ below $\aleph_\w$}\label{sec:basics}

For a set $A$ and a cardinal $\mu < |A|$, recall that $[A]^{\mu}$ denotes the set of all subsets of $A$ with cardinality $\mu$. A subset $\D$ of $[X]^{\mu}$ is \emph{cofinal} in $\hspace{-.1mm} \big( [A]^{\mu},\sub \hspace{-1mm} \big)$ if every member of $[A]^{\mu}$ is included in some member of $\D$, and we define
$$\mathrm{cf} \hspace{-.1mm} \big( [A]^{\mu},\sub \hspace{-1mm} \big) = \min \set{|\D|}{\D \text{ is cofinal in }\hspace{-.1mm} \big( [A]^{\mu},\sub \hspace{-1mm} \big)}.$$
For example, $\mathrm{cf} \hspace{-.1mm} \big( [\w_1]^{\aleph_0},\sub \hspace{-1mm} \big) = \aleph_1$, because (identifying an ordinal with the set of its predecessors, as usual) $\D = \set{\a}{\w \leq \a < \w_1}$ is cofinal in ${\hspace{-.1mm} \big( [\w_1]^{\aleph_0},\sub \hspace{-1mm} \big)}$.

Given a topological space $X$, recall that a \emph{cellular family} in $X$ is a collection $\mathcal S$ of pairwise disjoint, nonempty open subsets of $X$. The \emph{cellularity} of $X$, denoted $c(X)$, is defined as
$$c(X) = \sup \set{\card{\mathcal S}}{\mathcal S \text{ is a cellular family in }X}.$$ 
We will need the following two facts about cellularity and weight:

\begin{lemma}\label{lem:metricbasics}
If $X$ is a metric space of weight $\k$, then
\begin{enumerate}
\item $c(X) = \k$.
\item $X$ has a basis $\B$ of size $\k$ such that every $x \in X$ is contained in only countably many members of $\B$.
\end{enumerate}
\end{lemma}
\begin{proof}
For the first assertion, see \cite[Theorem 4.1.15]{Engelking}. 
For the second assertion, first recall that by the Bing Metrization Theorem, $X$ has a basis $\B_0$ that is a countable union of cellular families. Every $x \in X$ is contained in only countably many members of $\B_0$, and this remains true for any $\B \sub \B_0$. 
By \cite[Theorem 1.1.15]{Engelking}, $X$ has a basis $\B \sub \B_0$ with $|\B| = \k$.
\end{proof}

Another fact used frequently throughout the paper is: \emph{A subspace of a completely metrizable space is completely metrizable if and only if it is a $G_\delta$} (see \cite[Theorems 4.3.23 and 4.3.24]{Engelking}).

\begin{theorem}\label{thm:bounds}
Let $\k$ be an uncountable cardinal, and let $X$ be any completely metrizable space of weight $\k$. Then
$$\k \,\leq\, \cov(X) \,\leq\, \cofinal.$$
Furthermore, these bounds are attained: if $\k$ is given the discrete topology, then both $\k$ and $\k^\w$ are completely metrizable spaces of weight $\k$, $\cov(\k) = \k$, and $\cov(\k^\w) = {\cofinal}$.
\end{theorem}
\begin{proof}
Let $\k > \aleph_0$, and fix a completely metrizable space $X$ of weight $\k$.

Let $\C$ be a collection of Polish subspaces of $X$ with $|\C| < \k$. Because $c(X) = \k$ (by Lemma~\ref{lem:metricbasics}$(1)$), there is a cellular family $\mathcal S$ in $X$ with $|\C| < |\mathcal S|$. Each member of $\C$ has countable cellularity, and therefore each member of $\C$ meets only countably many members of $\mathcal S$. Hence $\bigcup \C$ meets at most $\aleph_0 \cdot |\mathcal C|$ members of $\mathcal S$. But $\C$ is uncountable (otherwise $X$ would be a countable union of separable subspaces, hence separable), so $\aleph_0 \cdot |\C| = |\C| < |\mathcal S|$. Thus there are $U \in \mathcal S$ with $U \cap \bigcup \C = \0$, which means that $\C$ does not cover $X$. Therefore $\cov(X) \geq \k$.

Applying Lemma~\ref{lem:metricbasics}$(2)$, let $\B$ be a basis for $X$ with $|\B| = \k$, such that every point of $X$ is contained in only countably many members of $\B$. Let $\D \sub [\B]^{\aleph_0}$ be cofinal in $\hspace{-.1mm} \big( [\B]^{\aleph_0},\sub \hspace{-1mm} \big)$, and for each $A \in \D$ define
$$Y_A = \set{x \in X}{A \supseteq \set{U \in \B}{x \in U}}.$$
It is straightforward to check that each $Y_A$ is a closed subspace of $X$, hence completely metrizable. Also, each of the $Y_A$ is second countable, with $\set{Y_A \cap U}{U \in A}$ as a basis. 
Therefore $\set{Y_A}{A \in \D}$ is a collection of Polish subspaces of $X$.
Also, $\set{Y_A}{A \in \D}$ covers $X$: if $x \in X$ then $\set{U \in \B}{x \in U}$ is countable by our choice of $\B$, and so $\set{U \in \B}{x \in U} \sub A$ for some $A \in \D$, by our choice of $\D$.
Hence $\cov(X) \leq \cf\hspace{-.1mm} \big( [\B]^{\aleph_0},\sub \hspace{-1mm} \big) = \cofinal$.

This completes the proof of the first part of the theorem, and it remains to prove the ``furthermore'' part.
It is clear that both $\k$ and $\k^\w$ (where $\k$ is given the discrete topology) are completely metrizable spaces of weight $\k$, and that $\cov(\k) = \k$. Also, it follows from what we already proved that $\cov(\k^\w) \leq \cofinal$. It remains to prove $\cov(\k^\w) \geq \cofinal$.

Let $\C$ be a covering of $\k^\w$ with Polish spaces, and for each $Y \in \C$ define 
$$A_Y = \set{x(n)}{x \in Y \text{ and } n \in \w}.$$
Here, as usual, we are identifying the points of $\k^\w$ with functions $\w \to \k$.
As each $Y \in \C$ has countable cellularity, $\set{x(n)}{x \in Y}$ is countable for each $n$; hence $\bigcup_{n < \w} \set{x(n)}{x \in Y} = A_Y \in [\k]^{\aleph_0}$ for every $Y \in \C$.
If $B \in [\k]^{\aleph_0}$, then let $x: \w \to B$ be an enumeration of $B$. Because $\C$ covers $\k^\w$, there is some $Y \in \C$ with $x \in Y$, and this implies $B \sub A_Y$.
Therefore $\set{A_Y}{Y \in \C}$ is cofinal in $[\k]^{\aleph_0}$.
Because $\C$ was an arbitrary covering of $\k^\w$ with Polish spaces, this implies $\cofinal \leq \cov(\k^\w)$.
\end{proof}

We observed above that $\mathrm{cf} \hspace{-.1mm} \big( [\w_1]^{\aleph_0},\sub \hspace{-1mm} \big) = \aleph_1$ because $\set{\a}{\w \leq \a < \w_1}$ is cofinal in $\big( [\w_1]^{\aleph_0},\sub \hspace{-1mm} \big)$.
Similarly, $\mathrm{cf} \hspace{-.1mm} \big( [\w_n]^{\aleph_{n-1}},\sub \hspace{-1mm} \big) = \aleph_n$ for all $n > 0$, because $\set{\a}{\w_{n-1} \leq \a < \w_n}$ is cofinal in $\big( [\w_n]^{\aleph_{n-1}},\sub \hspace{-1mm} \big)$.
By induction, this can be used to show
$\mathrm{cf} \hspace{-.1mm} \big( [\w_n]^{\aleph_0},\sub \hspace{-1mm} \big) = \aleph_n$ 
for all $n > 0$.
Combining this equality with Theorem~\ref{thm:bounds}, we obtain the following:

\begin{corollary}\label{cor:alephncov}
Let $X$ be a completely metrizable space and suppose that $\wt(X) = \aleph_n$ for some $n \in \w \setminus \{0\}$. Then
$\cov(X) = \aleph_n$.
\end{corollary}

More generally, for any cardinal $\k$ with $\k = \cofinal$, $\cov(X) = \k$ for every completely metrizable space $X$ of weight $\k$. 
This is fairly informative: if $0^\dagger$ does not exist, then $\k = \cofinal$ for all $\k$ with $\cf(\k) > \w$. This is discussed further in Section~\ref{sec:L}.
On the other hand, if $\k$ is uncountable and $\cf(\k) = \w$, then $\cofinal > \k$. 
(This is proved by a simple diagonal argument. 
If $\cofinal \leq \k$, we would have a cofinal subset $\bigcup_{n < \w}\D_n$ of $[\k]^{\aleph_0}$ with $|\D_n| < \k$ for all $n$, but then choosing $\a_n \in \k \setminus \bigcup \D_n$ for all $n$, $\set{\a_n}{n \in \w}$ would not be contained in any member of any $\D_n$.) 

The situation for $\mp$ is different.
We shall see in the next section that $\mp(X) > \mathrm{cf} \hspace{-.1mm} \big( [\w_\w]^{\aleph_0},\sub \hspace{-1mm} \big)$ is consistent, relative to a huge cardinal, for certain completely metrizable spaces $X$ of weight $\aleph_\w$. 
In particular, the upper bound for $\cov(X)$ in Theorem~\ref{thm:bounds} does not necessarily apply to $\mp(X)$.
Nonetheless, we now show, via an inductive argument somewhat similar to that preceding Corollary~\ref{cor:alephncov}, that $\mp(X) = \wt(X)$ whenever $\wt(X) < \aleph_\w$.

\begin{theorem}\label{thm:induct}
Let $\k$ be a cardinal with uncountable cofinality, and let $X$ be a completely metrizable space of weight $\k$.
There is a size-$\k$ partition of $X$ into completely metrizable spaces of weight $<\!\k$.
\end{theorem}
\begin{proof}
Using Lemma~\ref{lem:metricbasics}$(2)$, let $\B = \set{U_\a}{\a < \k}$ be a basis for $X$ such that each $x \in X$ is contained in only countably many members of $\B$.
For each $\a < \k$, define 
$$X_\a = \set{x \in X}{\sup \set{\xi < \k}{x \in U_\xi} \leq \a} \quad \text{and} \quad Y_\a = X_\a \setminus \textstyle \bigcup_{\xi < \a}X_\xi.$$
By our choice of $\B$, $\set{\xi < \k}{x \in U_\xi}$ is countable for each $x \in X$. Because $\cf(\k) > \w$, this implies $\bigcup_{\a < \k}X_\a = X$. It follows that $\set{Y_\a}{\a < \k}$ is a partition of $X$.
Also, $\set{X_\a \cap U_\xi}{\xi \leq \a}$ is a basis for $X_\a$, and therefore $\wt(Y_\a) \leq \wt(X_\a) \leq |\a| < \k$ for each $\a < \k$.

To finish the proof, we must show that each of the $Y_\a$ is completely metrizable.
It is not difficult to see that each of the $X_\a$ is closed in $X$, and that the $X_\a$ are increasing, i.e., $\b < \a$ implies $X_\b \sub X_\a$.
If $\a = \b+1$ is a successor ordinal, then $Y_\a = X_\a \setminus X_\b$ (because the $X_\a$ are increasing). It follows that $Y_\a$ is a $G_\dlt$ subset of $X$, hence completely metrizable.
If $\a$ is a limit ordinal with $\cf(\a) = \w$, fix an increasing sequence $\seq{\b_n}{n \in \w}$ with limit $\a$. Then $Y_\a = X_\a \setminus \bigcup_{n \in \w}X_{\b_n}$, so that once again $Y_\a$ is a $G_\dlt$ in $X$, hence completely metrizable.
Finally, if $\a$ is a limit ordinal with $\cf(\a) > \w$, then $Y_\a = \0$, because by our choice of $\B$, $\sup \set{\xi < \k}{x \in U_\xi}$ cannot have uncountable cofinality, and so $x \in X_\a$ implies $x \in X_\xi$ for some $\xi < \a$.
\end{proof}

\begin{theorem}\label{thm:alephnpar}
Let $X$ be a completely metrizable space and suppose that $\wt(X) = \aleph_n$ for some $n \in \w \setminus \{0\}$. Then
$\cov(X) = \mp(X) = \aleph_n$.
\end{theorem}
\begin{proof}
That $\cov(X) = \aleph_n$ is given by Corollary~\ref{cor:alephncov}. That $\mp(X) \leq \aleph_n$ follows from the previous theorem and a straightforward induction, and the reverse inequality is given in Theorem~\ref{thm:bounds}.
\end{proof}

\begin{corollary}\label{cor:c}
Suppose that there is a completely metrizable space $X$ with $\cov(X) < \mp(X)$. Then $\continuum \geq \aleph_{\w+1}$.
\end{corollary}
\begin{proof}
Suppose $\continuum < \aleph_{\w+1}$. Then $\continuum < \aleph_\w$, because $\continuum$ cannot have countable cofinality. If $\wt(X) \leq \continuum$, then $\cov(X) = \mp(X)$ by the previous theorem. If $\wt(X) > \continuum$, then $|X| > \continuum$ and (as pointed out in the introduction) this implies $\cov(X) = \mp(X) = |X|$.
\end{proof}

\begin{corollary}\label{cor:image}
Let $X$ completely metrizable, with $\k = \wt(X) < \aleph_\w$.
If there is a continuous bijection $X \to [0,1]$, then there is a partition of $[0,1]$ into $\k$ Borel sets.
\end{corollary}
\begin{proof}
Suppose $f: \k^\w \to [0,1]$ is a continuous bijection. 
By the previous theorem, there is a partition $\P$ of $\k^\w$ into $\k$ Polish spaces. 
By a theorem of Lusin and Suslin \cite[Theorem 15.1]{Kechris}, if $B \sub [0,1]$ is a continuous bijective image of a Polish space, then $B$ is Borel.
Hence $\set{f[X]}{X \in \P}$ is a partition of $[0,1]$ into Borel sets.
\end{proof}

Theorems \ref{thm:induct} and \ref{thm:alephnpar} are reminiscent of Theorem 3.5 and Corollary 3.6 in \cite{Brian&Miller}. In fact, these results from \cite{Brian&Miller} motivated much of the present paper. They state, respectively (giving all ordinals the discrete topology):
\begin{itemize}
\item[$(3.5)$] For each ordinal $\a$, the space $\w_{\a+1}^\w$ can be partitioned into $\aleph_{\a+1}$ copies of $\w_\a^\w$.
\item[$(3.6)$] For each $n \in \w \setminus \{0\}$, $\w_n^\w$ can be partitioned into $\aleph_n$ copies of the Baire space $\w^\w$.
\end{itemize}
We note in passing that $(3.5)$ and $(3.6)$ can be seen as special cases of Theorems~\ref{thm:induct} and \ref{thm:alephnpar}, respectively. Thus to some extent, the results of this section subsume these results from \cite{Brian&Miller}. 
To deduce $(3.5)$ and $(3.6)$ from Theorems~\ref{thm:induct} and \ref{thm:alephnpar}, all one needs is the following lemma (whose proof we omit): \emph{If $X$ is a zero-dimensional, completely metrizable space of weight $\leq\!\k$, then $X \times \k^\w \homeo \k^\w$.}
Armed with this lemma, we see that if $\P$ is a partition of $\w_{\a+1}^\w$ as described in Theorem~\ref{thm:induct}, then $\set{Z \times \w_\a^\w}{Z \in \P}$ is a partition of $\w_{\a+1}^\w \times \w_\a^\w \homeo \w_{\a+1}^\w$ into $\aleph_{\a+1}$ copies of $\w_\a^\w$.
Similarly, if $\P$ is a partition of $\w_n^\w$ into $\aleph_n$ Polish spaces as in Theorem~\ref{thm:alephnpar}, then $\set{Z \times \w^\w}{Z \in \P}$ is a partition of $\w_n^\w \times \w^\w \homeo \w_n^\w$ into $\aleph_n$ copies of $\w^\w$.

\section{A model for $\cov < \mp$}\label{sec:<}

In this section, every ordinal, when referred to as a topological space, carries the discrete topology.
We write $\w_\w^\w$ rather than $(\w_\w)^\w$ for the countable power of the size-$\aleph_\w$ discrete space $\w_\w$.

The main result of this section is that it is consistent relative to a huge cardinal that $\cov(\w_\w^\w) < \mp(\w_\w^\w)$. The space $\w_\w^\w$ has weight $\aleph_\w$, so this result is optimal in some sense: by Theorem~\ref{thm:alephnpar}, completely metrizable spaces $X$ of smaller weight must have $\cov(X) = \mp(X)$. 


Recall that $\mathrm{cov}(\mathcal M)$ denotes the smallest cardinality of a collection $\F$ of meager subsets of the Baire space $\w^\w$ with $\bigcup \F = \w^\w$. The use of $\w^\w$ in this definition is inessential: if $Y$ is any Polish space without isolated points, then $\mathrm{cov}(\mathcal M)$ is the smallest cardinality of a collection $\F$ of meager subsets of $Y$ with $\bigcup \F = Y$. (For a proof that these two ways of defining $\mathrm{cov}(\mathcal M)$ really are equivalent, see \cite[Proposition 2]{Fremlin&Shelah}.) 

Suppose $X$ is a completely metrizable space, and $\P$ is a partition of $X$ into completely metrizable subspaces. 
We say that $\P$ is \emph{skinny} if for every Polish $Y \sub X$, $\set{Z \in \P}{Z \cap Y \neq \0}$ is countable.

\begin{lemma}\label{lem:covm}
Suppose $X$ is a completely metrizable space, and $\P$ is a partition of $X$ into completely metrizable subspaces. If $\P$ is not skinny, then $|\P| \geq \mathrm{cov}(\mathcal M)$.
\end{lemma}
\begin{proof}
This follows almost immediately from a result of Fremlin and Shelah (which answered a question going back to Hausdorff) \cite[Theorem 3]{Fremlin&Shelah}: \emph{Any partition of a Polish space into uncountably many $G_\delta$ sets is of size $\geq \mathrm{cov}(\mathcal M)$.}

Suppose $X$ is a completely metrizable space, and $\P$ is a partition of $X$ into completely metrizable subspaces.
If $\P$ is not skinny, there is some Polish $Y \sub X$ such that $\set{Z \in \P}{Y \cap Z \neq \0}$ is uncountable.
But then $\set{Y \cap Z}{Z \in \P} \setminus \{\0\}$ is an uncountable partition of $Y$ into $G_\delta$ sets, so it follows from Fremlin and Shelah's theorem that $|\P| \geq \mathrm{cov}(\mathcal M)$. 
\end{proof}

Consider the following statement, abbreviated $(\k^+,\k) \hspace{-.5mm} \twoheadrightarrow \hspace{-.5mm} (\mu^+,\mu)$:
\begin{itemize}
\item[$\ $] For every model $M$ for a countable language $\mathcal L$ that contains a unary predicate $A$, if $|M| = \k^+$ and $|A| = \k$ then there is an elementary submodel $M' \prec M$ such that $|M'| = \mu^+$ and $|M' \cap A| = \mu$.
\end{itemize}
The statement $(\k^+,\k) \hspace{-.5mm} \twoheadrightarrow \hspace{-.5mm} (\mu^+,\mu)$ is an instance of \emph{Chang's conjecture}. To prove the main theorem of this section, we use the statement obtained by taking $\k = \aleph_\w$ and $\mu = \aleph_0$ above. This instance of Chang's conjecture, abbreviated \chang, is known as \emph{Chang's conjecture for $\aleph_\w$}.

The usual Chang conjecture, which is the assertion $(\aleph_2,\aleph_1) \hspace{-.5mm} \twoheadrightarrow \hspace{-.5mm} (\aleph_1,\aleph_0)$, is equiconsistent with the existence of an $\w_1$-Erd\H{o}s cardinal. Chang's conjecture for $\aleph_\w$ requires even larger cardinals. \gch + \chang was first proved consistent relative to a hypothesis a little weaker than the existence of a $2$-huge cardinal in \cite{LMS}. Recently this was improved to a huge cardinal in \cite{EH}. The precise consistency strength of \chang is an open problem, but significant large cardinal strength is known to be needed. This is because \chang implies the failure of $\square_{\aleph_\w}$ (see \cite{SV}, in particular Fact 4.2 and the remarks after it), and the failure of $\square_{\aleph_\w}$ carries significant large cardinal strength (see \cite{CF}).

\begin{lemma}\label{lem:cohen}
\chang implies that no partition of $\w_\w^\w$ into Polish spaces is skinny.
\end{lemma}
\begin{proof}
Let $\P$ be a pairwise disjoint collection of Polish subspaces of $\w_\w^\w$. By Theorem~\ref{thm:bounds}, $\cov(\w_\w^\w) = \mathrm{cf} \hspace{-.1mm} \big( [\w_\w]^{\aleph_0},\sub \hspace{-1mm} \big)$, and a straightforward diagonal argument shows $\mathrm{cf} \hspace{-.1mm} \big( [\w_\w]^{\aleph_0},\sub \hspace{-1mm} \big) > \aleph_\w$. Hence $|\P| \geq \mp(\w_\w^\w) \geq \cov(\w_\w^\w) \geq \aleph_{\w+1}$. For each $Y \in \P$, let 
$$A_Y = \set{x(n)}{x \in Y \text{ and }n \in \w}.$$
As in the proof of Theorem~\ref{thm:bounds}, $A_Y$ is countable for each $Y \in \P$.

Let $(M,\in)$ be a model of (a sufficiently large fragment of) \zfc such that $\w_\w \sub M$, $\P \in M$, and $|M| = |M \cap \P| = \aleph_{\w+1}$. (Such a model can be obtained in the usual way, via the downward L\"{o}wenheim-Skolem Theorem.) Let $\phi: M \to M \cap \P$ be a bijection, and consider the model $(M,\in,\phi,\w_\w)$ for the $3$-symbol language consisting of a binary relation, a unary function, and a unary predicate. Applying the Chang conjecture \chang, there exists some $M' \sub M$ such that $|M'| = \aleph_1$, $M' \cap \w_\w$ is countable, and $(M',\in,\phi,\w_\w) \prec (M,\in,\phi,\w_\w)$.

Let $\P' = \P \cap M'$. By elementarity, the restriction of $\phi$ to $M'$ is a bijection $M' \to \P'$, and so $|\P'| = \aleph_1$.

Let $A = \w_\w \cap M'$. If $Y \in \P'$, then $A_Y \in M'$, and therefore $A_Y \sub M'$ (because $A_Y$ is countable, and $M'$ models (enough of) \zfc). Hence $Y \in \P'$ implies $A_Y \sub A$. Let
$X = A^\w$.
Then $X \homeo \w^\w$ (in particular, $X$ is Polish), and $X \supseteq A_Y^\w \supseteq Y$ for all $Y \in \P'$. 
In particular, $\set{Y \in \P}{Y \cap X \neq \0} \supseteq \P'$ is uncountable, so $\P$ is not skinny.
\end{proof}

\begin{theorem}\label{thm:huge}
It is consistent relative to a huge cardinal that $\cov(\w_\w^\w) < \mp(\w_\w^\w)$.
More precisely, given a model of $\gch+{}$\chang and any $\lambda > \aleph_{\w+1}$ with $\cf(\lambda) > \w$, there is a ccc forcing extension in which
$\aleph_{\w+1} \,=\, \cov(\w_\w^\w) \,<\, \mp(\w_\w^\w) \,=\, \lambda \,=\, \continuum.$
\end{theorem}
\begin{proof}
Let $V$ be a model of \gch plus \chang. Recall that the existence of such a model is consistent relative to a huge cardinal.

Let $\PP$ be any ccc forcing poset such that $\forces_\PP \mathrm{cov}(\mathcal M) = \continuum = \lambda$. For example, we could take $\PP = \mathrm{Fn}(\lambda,2)$, the standard poset for adding $\lambda$ Cohen reals, or $\PP$ could be the standard length-$\lambda$ iteration forcing $\ma + \continuum = \lambda$.

Because $\PP$ has the ccc, $\PP$ preserves cardinals; hence $\forces_\PP \aleph_{\w+1} < \lambda = \continuum$. 
Also, $\forces_\PP \mathrm{cf} \hspace{-.1mm} \big( [\w_\w]^{\aleph_0},\sub \hspace{-1mm} \big) = \aleph_{\w+1}$ because $\PP$ is ccc and the ground model satisfies \gch.
(This fact is fairly well known, but we include a sketch of the argument here for completeness. By the comments following Corollary~\ref{cor:alephncov}, $\mathrm{cf} \hspace{-.1mm} \big( [\w_\w]^{\aleph_0},\sub \hspace{-1mm} \big) > \aleph_\w$. 
Because $\PP$ has the ccc, every member of $[\w_\w]^{\aleph_0}$ in the extension is contained in some member of $[\w_\w^\w]^{\aleph_0} \cap V$. That is, $[\w_\w]^{\aleph_0} \cap V$ is cofinal in $[\w_\w]^{\aleph_0}$. Because \gch holds in $V$, this means
$\mathrm{cf} \hspace{-.1mm} \big( [\w_\w]^{\aleph_0},\sub \hspace{-1mm} \big) \leq \card{[\w_\w]^{\aleph_0} \cap V} = \aleph_{\w+1}$.)
By Theorem~\ref{thm:bounds}, 
$\forces_\PP \cov(\w_\w^\w) = \aleph_{\w+1}$.

Furthermore, \chang is preserved by ccc forcing. (This fact is considered folklore, but a proof can be found in \cite[Lemma 13]{EH}.)
Hence $\forces_\PP \mp(\w_\w^\w) \geq \mathrm{cov}(\mathcal M) = \continuum$ by the previous two lemmas.
Finally, $\mp(\w_\w^\w) \leq |\w_\w^\w| = |\w_\w^\w|^{\aleph_0} \leq \continuum^{\aleph_0} = \continuum$ in the extension (because we may partition $\w_\w^\w$ into singletons).  
\end{proof}

\begin{corollary}
It is consistent relative to a huge cardinal that $\continuum$ is arbitrarily large, and $\cov(X) < \mp(X)$
for every completely metrizable space $X$ with $\aleph_\w \leq \wt(X) < \continuum$ such that $X$ contains a subspace homeomorphic to $\w_\w^\w$.
\end{corollary}
\begin{proof}
In statement of Theorem~\ref{thm:huge}, let us add the requirement that $\lambda$ is not the successor of a singular cardinal with cofinality $\w$.

In the generic extension, let $X$ be a completely metrizable space with $\wt(X) = \k$, where $\aleph_\w \leq \k < \lambda = \continuum$.
By Theorem~\ref{thm:bounds}, $\cov(X) \leq \cofinal$. As in the proof of Theorem~\ref{thm:huge} above, $[\k]^{\aleph_0} \cap V$ is cofinal in $\mathrm{cf} \hspace{-.1mm} \big( [\k]^{\aleph_0},\sub \hspace{-1mm} \big)$, and therefore $\cofinal \leq \card{[\k]^{\aleph_0} \cap V} = (\k^{\aleph_0})^V < \lambda$. (The last inequality uses the fact that \gch holds in $V$, plus the assumption that $\lambda$ is not the successor of a singular cardinal with cofinality $\w$.) Hence $\cov(X) < \continuum$. But if $Y$ is any subspace of $X$ homeomorphic to $\w_\w^\w$ and $\P$ is any partition of $X$ into Polish spaces, then $\set{Z \cap Y}{Z \in \P}$ is a partition of $Y$ into Polish spaces. Therefore $\mp(X) \geq \mp(Y) = \mp(\w_\w^\w) = \continuum$.
\end{proof}

This raises the question of what completely metrizable spaces with weight $\geq\!\aleph_\w$ contain a copy of $\w_\w^\w$. A characterization is given at the beginning of the next section: roughly, $X$ contains a (closed) copy of $\w_\w^\w$ unless $X$ and all its closed subsets are everywhere ``locally'' of weight $<\!\aleph_\w$.

\vspace{2mm}

A collection $\F$ of countable sets is called \emph{sparse} if $\bigcup \G$ is uncountable for every uncountable $\G \sub \F$. Equivalently, $\F$ is sparse if no single countable set contains uncountably many members of $\F$.

Sparse cofinal subsets of $\hspace{-.1mm} \big( [\w_\w]^{\aleph_0},\sub \hspace{-1mm} \big)$ have appeared in diverse problems from infinite combinatorics: 
a good deal about them can be found in Kojman, Milovich, and Spadaro's \cite{KMS}, and they appear also in Spadaro \cite{Spadaro}, Blass \cite{Blass}, and Nyikos \cite{Nyikos}. (Nyikos refers to them as \emph{Kuprepa families}.) It is known that \chang implies the non-existence of sparse cofinal families in $\hspace{-.1mm} \big( [\w_\w]^{\aleph_0},\sub \hspace{-1mm} \big)$, and $\square_{\aleph_\w}$ implies they do exist \cite[Section 3]{KMS}.

\begin{observation}
If there is a skinny partition of $\w_\w^\w$ into Polish spaces, then there is a sparse cofinal subset of ${\hspace{-.1mm} \big( [\w_\w]^{\aleph_0},\sub \hspace{-1mm} \big)}$.
\end{observation}
\begin{proof}
Suppose $\P$ is a partition of $\w_\w^\w$ into Polish spaces. In the proof of Theorem~\ref{thm:bounds}, we showed that this implies 
$$\D = \set{\set{x(n)}{x \in X \text{ and }n \in \w}}{X \in \P}$$
is a cofinal subset of ${\hspace{-.1mm} \big( [\w_\w]^{\aleph_0},\sub \hspace{-1mm} \big)}$. 

Suppose $\D$ is not sparse. Then there is some $A \in [\w_\w]^{\aleph_0}$ such that $\set{B \in \D}{B \sub A}$ is uncountable. By the definition of $\D$, this means that $\set{X \in \P}{X \sub A^\w}$ is uncountable, and in particular $\set{X \in \P}{X \cap A^\w \neq \0}$ is uncountable.
\end{proof}

In other words, a skinny partition of $\w_\w^\w$ is just a special kind of sparse cofinal family, in disguise. In light of this, our use of \chang in this section, and our use of weak $\square$-like principles in the next, are unsurprising. 
We note (without proof) that the partitions whose existence is implied in the proof of Theorem~\ref{thm:alephnpar} or in the next section, which witness ``small'' values of $\mp(X)$, are skinny, and therefore naturally give rise to sparse cofinal families.

\section{$L$-like principles imply $\cov = \mp$}\label{sec:L}

In this section we show that 
if $\cov(X) < \mp(X)$ for any completely metrizable space $X$, then $0^\dagger$ exists. We do this by extending the inductive arguments of Section~\ref{sec:basics} to spaces of weight $\geq\!\aleph_\w$ using ``$L$-like" assumptions that hold if $0^\dagger$ does not exist. Roughly, the assumptions we use are that the cardinals $\cofinal$ are as small as possible, and that a weak $\square$ principle holds at singular cardinals $<\!\continuum$ of cofinality $\w$.

First we need a purely topological theorem that makes no assumptions beyond \zfc.

Let us say that a space $X$ is \emph{locally $<\! \k$-like} if for every nonempty open $U \sub X$, there is a nonempty open $V \sub U$ with $\wt(V) < \k$.
Equivalently, $X$ is locally $<\!\k$-like if $X$ has a $\pi$-base of open sets with weight $<\!\k$.

On the other hand, let us say that a space $X$ is \emph{locally $\geq\! \k$-like} if every nonempty open $U \sub X$ has $\wt(U) \geq \k$. Note that $X$ fails to be locally $<\!\k$-like if and only if some nonempty open subset of $X$ is locally $\geq\!\k$-like.

Compare the following theorem with Theorem~\ref{thm:induct}.

\begin{theorem}\label{thm:dichotomy}
Let $\k$ be an uncountable cardinal with $\cf(\k) = \w$, and let $X$ be a completely metrizable space with $\wt(X) \geq \k$.
Then either:
\begin{enumerate}
\item some closed subset of $X$ is homeomorphic to $\k^\w$, or

\vspace{1mm}

\item there is a partition of $X$ into at most $\wt(X)$ completely metrizable spaces, each of weight $<\!\k$.
\end{enumerate}
Furthermore, these two possibilities are mutually exclusive, and $(2)$ holds if and only if every closed subset of $X$ is locally $<\!\k$-like.
\end{theorem}
\begin{proof}
To prove the theorem, we show: 
\begin{itemize}
\item[$(a)$] If some closed $Y \sub X$ is not locally $<\!\k$-like, then some closed $Z \sub Y$ is homeomorphic to $\k^\w$.
\item[$(b)$] If every closed subspace of $X$ is locally $<\!\k$-like, then $X$ can be partitioned into at most $\wt (X)$ completely metrizable spaces, each of weight $<\!\k$.
\end{itemize}

We begin with $(a)$. 
Suppose $X_0 \sub X$ is closed and is not locally $<\!\k$-like. This means there is a nonempty open $X_1 \sub X_0$ (open in $X_0$, not necessarily in $X$) that is locally $\geq\!\k$-like.
Let $Y = \closure{X_1}$. Note that $Y$ is locally $\geq\!\k$-like because $X_1$ is; in fact, if $U$ is any nonempty open subset of $Y$, then $\closure{U}$ is locally $\geq\!\k$-like.

Recall that a collection $\C$ of subsets of some topological space is \emph{discrete} if every point has a neighborhood meeting $\leq\!1$ member of $\C$.

\begin{claim} 
Suppose $Y$ is a complete metric space and is locally $\geq\!\k$-like, and let $\e > 0$. There is a discrete collection of $\k$ open subsets of $Y$, each with diameter $\leq\!\e$.
\end{claim}

\vspace{1mm}

\noindent \emph{Proof of claim:}
By the Bing metrization theorem, $Y$ has a $\s$-discrete base: that is, there is a sequence $\seq{\B_n}{n \in \w}$ of collections of open subsets of $Y$ such that $\bigcup_{n \in \w}\B_n$ is a base for $Y$ and each $\B_n$ is discrete. By deleting any member of any $\B_n$ with diameter $>\!\e$, we may (and do) assume without loss of generality that each member of each $\B_n$ has diameter $\leq\!\e$.

If some $\B_n$ has size $\geq\!\k$, we are done; so suppose this is not the case. Fix a discrete collection $\set{U_k}{k \in \w}$ of open subsets of $Y$ (e.g., any countable subset of some $\B_n$).
For each $k$ and $n$, let $\B_n^k = \set{U_k \cap V}{V \in \B_n} \setminus \{\0\}$, and note that $\bigcup_{n \in \w} \B_n^k$ is a basis for $U_k$.
Fix a sequence $\seq{\mu_n}{n \in \w}$ of cardinal increasing up to $\k$.
Because $\wt(U_k) \geq \k$, there is some $n_k \in \w$ such that $|\B^k_{n_k}| \geq \mu_k$.
For each $k$, let $\C_k = \set{\closure{V}}{V \in \B^k_{n_k}}$, and let $\C = \bigcup_{k \in \w}\C_k$.
\hfill ${\footnotesize \qed} \ $

\vspace{2mm}

By recursion, we now construct a tree $\set{V_s}{s \in \k^{<\w}}$ of closed subsets of $Y$ as follows.
Let $V_\0 = Y$.
Given $s \in \k^{<\w}$, suppose some open set $V_s$ has already been chosen.
Applying the claim above, let $\set{U_{s \cat \a}}{\a < \k}$ be a discrete collection of open subsets of $V_s$, each of diameter $\leq \nicefrac{1}{2^{|s|}}$.
Finally, let $V_{s \cat \a} = \closure{U}_{s \cat \a}$ for each $\a < \k$.
This completes the recursion.

For each $x \in \k^\w$, the completeness of our metric, together with the size restriction on the $V_s$, implies that $\bigcap_{n \in \w}V_{x \restriction n}$ contains exactly one point.
Define $h: \k^\w \to Y$ by taking $h(x)$ to be the unique point in $\bigcap_{n \in \w} V_{x \restriction n}$ for each $x \in \k^\w$.

We claim $h$ is a topological embedding.
For injectivity, if $x \neq y$ then $x \rest n \neq y \rest n$ for some $n$, and $x$ and $y$ are members of the disjoint sets $V_{x \restriction n}$ and $V_{y \restriction n}$. Continuity follows from our size restriction on the $V_s$: if $\e > 0$ and $x \in \w_\w^\w$, there is some $n$ such that $V_{x \restriction n} \sub B_\e(h(x))$. Finally, $h$ is a closed mapping: if $\tr{s} = \set{x \in \k^\w}{s \sub x}$ is a basic closed set in $\k^\w$, then $h(\tr{s}) = V_s \cap h(\k^\w)$.
This completes the proof of $(a)$.

To prove $(b)$, suppose every closed $K \sub X$ is locally $<\!\k$-like.
Let $W$ be a function choosing witnesses to this property: i.e., if $K$ is a nonempty closed subset of $X$, then $W(K)$ is a nonempty open subset of $K$ with weight $<\!\k$.

Let $\lambda = \wt(X)$.
Define a decreasing sequence $\seq{X_\a}{\a < \lambda^+}$ of closed subspaces of $X$ as follows:
let $X_0 = X$,
let $\textstyle X_\a = \bigcap_{\xi < \a} X_\xi$ for limit $\a < \lambda^+$, and 
at successor stages let $X_{\a+1} = X_\a \setminus W(X_\a)$.

Because $\wt(X) = \lambda$, there is some $\a < \lambda^+$ such that $X_\b = X_\a$ for all $\b \geq \a$.
But this implies $X_\b = \0$ for all $\b \geq \a$, since our construction would otherwise give $W(X_\a) \neq \0$ and therefore $X_{\a+1} \neq X_\a$.

Each $W(X_\xi)$ is an open subspace of a closed subspace of $X$, and therefore is completely metrizable. Furthermore $\wt(W(X_\xi)) < \k$ for all $\xi$, and $\set{W(X_\xi)}{\xi < \a}$ is a partition of $X \setminus X_\a = X$.
\end{proof}

\begin{corollary}\label{cor:upperbound}
Let $\k$ be an uncountable cardinal with $\cf(\k) = \w$, and let $X$ be a completely metrizable space with $\wt(X) = \k$. If some closed subset of $X$ is not locally $<\!\k$-like, then $\cov(X) = \cofinal$.
\end{corollary}
\begin{proof}
Let $Z$ be a closed subspace of $X$ homeomorphic to $\k^\w$.
If $\C$ is a covering of $X$ with Polish spaces, then $\set{Y \cap Z}{Y \in \C}$ is a covering of $Z$ with Polish spaces. Hence $\cov(X) \geq \cov(Z) = \cov(\k^\w) = \cofinal$. (The last equality is from Theorem~\ref{thm:bounds}.) But we also have $\cov(X) \leq \cofinal$ by Theorem~\ref{thm:bounds}, so $\cov(X) = \cofinal$. 
\end{proof}

\begin{corollary}\label{cor:omega}
Let $X$ be a completely metrizable space with $\wt(X) = \aleph_\w$.
Then either:
\begin{enumerate}
\item $\cov(X) = \cov(\w_\w^\w)$ and $\mp(X) \geq \mp(\w_\w^\w)$.

\vspace{1mm}

\item $\cov(X) = \mp(X) = \aleph_\w$.
\end{enumerate}

\end{corollary}
\begin{proof}
We show that statements $(1)$ and $(2)$ from Theorem~\ref{thm:dichotomy} imply, respectively, the statements $(1)$ and $(2)$ in the statement of the corollary.

Suppose statement $(1)$ from Theorem~\ref{thm:dichotomy} holds, and let $Z$ be a closed subspace of $X$ homeomorphic to $\w_\w^\w$.
By the previous corollary, $\cov(X) = \cov(\w_\w^\w) = \mathrm{cf} \hspace{-.1mm} \big( [\w_\w]^{\aleph_0},\sub \hspace{-1mm} \big)$.
If $\P$ is a partition of $X$ into Polish spaces, then $\set{P \cap Z}{P \in \P}$ is a partition of $Z$ into Polish spaces, and this shows that $\mp(X) \geq \mp(Z) = \mp(\w_\w^\w)$. 

Suppose statement $(2)$ from Theorem~\ref{thm:dichotomy} holds, and let $\P$ be a partition of $X$ into $\leq\! \aleph_\w$ completely metrizable spaces of weight $<\!\aleph_\w$.
For each $Y \in \P$, there is by Theorem~\ref{thm:alephnpar} a partition $\P_Y$ of $P$ into $<\! \aleph_\w$ Polish spaces.
Then $\bigcup \set{\P_Y}{Y \in \P}$ is a partition of $X$ into $\leq\! \aleph_\w$ Polish spaces. 
Therefore $\cov(X) \leq \mp(X) \leq \aleph_\w$. The reverse inequalities are provided by Theorem~\ref{thm:bounds}.
\end{proof}

The next theorem improves Corollary~\ref{cor:c} from Section~\ref{sec:basics}.

\begin{theorem}
Suppose $\cov(X) < \mp(X)$ for some completely metrizable space $X$.
Then $\continuum \geq \aleph_{\w+2}$.
\end{theorem}
\begin{proof}
If $\wt(X) < \aleph_\w$ then $\cov(X) = \mp(X) = \wt(X)$ by Theorem~\ref{thm:alephnpar}. If $\wt(X) > \continuum$ then $|X| > \continuum$ (by \cite[Theorem 4.1.15]{Engelking}), so $\cov(X) = \mp(X) = |X|$.

Now suppose $\wt(X) = \continuum$.
Because $\continuum \leq \cf \hspace{-.1mm} \big( [\continuum]^{\aleph_0},\sub \hspace{-1mm} \big) \leq \continuum^{\aleph_0} = \continuum$, Theorem~\ref{thm:bounds} gives $\cov(X) = \continuum$.
Also, $\continuum \leq |X| \leq \continuum^{\aleph_0} = \continuum$ (generally speaking, $\wt(Y) \leq |Y| \leq \wt(Y)^{\aleph_0}$ for any metric space $Y$). Therefore $|X| = \continuum$, and this implies $\mp(X) \leq \continuum$ (because we may partition $X$ into singletons). Hence $\continuum = \cov(X) \leq \mp(X) \leq \continuum$, and in particular $\cov(X) = \mp(X)$.

Therefore, if there is a completely metrizable space $X$ with $\cov(X) < \mp(X)$, then $\aleph_\w \leq \wt(X) < \continuum$. To prove the theorem, it suffices to show that if $\wt(X) = \aleph_\w$ and $\continuum = \aleph_{\w+1}$, then $\cov(X) = \mp(X)$.

Assume $\wt(X) = \aleph_\w$ and $\continuum = \aleph_{\w+1}$. By Corollary~\ref{cor:omega}, we may also assume $\cov(X) = \cov(\w^\w_\w)$ and $\mp(X) \geq \mp(\w_\w^\w)$.
But $|X| \leq \aleph_\w^{\aleph_0} \leq \continuum^{\aleph_0} = \continuum$, hence $\mp(X) \leq \continuum$ (by partitioning $X$ into singletons). Therefore
$\aleph_{\w+1} \leq \cf \hspace{-.1mm} \big( [\w_\w]^{\aleph_0},\sub \hspace{-1mm} \big) = \cov(\w_\w^\w) = \cov(X) \leq \mp(X) \leq \continuum = \aleph_{\w+1}$.
\end{proof}

The $\aleph_{\w+2}$ bound given in this theorem is optimal, which can be seen by setting $\lambda = \aleph_{\w+2}$ in Theorem~\ref{thm:huge}.

Corollary~\ref{cor:omega} gives us our first inkling of how to extend the results of Section~\ref{sec:basics} to spaces of weight $\aleph_\w$ and beyond. 
If we assume $\mathrm{cf} \hspace{-.1mm} \big( [\w_\w]^{\aleph_0},\sub \hspace{-1mm} \big) = \aleph_{\w+1}$ and we can somehow prove that $\mp(X) \leq \aleph_{\w+1}$ whenever $\wt(X) = \aleph_\w$, then using Corollary~\ref{cor:omega} we can conclude that $\cov(X) = \mp(X)$ whenever $\wt(X) = \aleph_\w$. 
This is precisely our strategy.
Proving $\mp(X) \leq \aleph_{\w+1}$ whenever $\wt(X) = \aleph_\w$, and similarly for other cardinals of cofinality $\w$, is where $\square$-like principles enter the picture.

Let $\schh$ abbreviate the following statement: 
\begin{itemize}
\item[$\mathsf{SCH}^+\!:$] If $\k$ is a cardinal with uncountable cofinality, then $\cofinal = \k$. If $\k$ is a singular cardinal with $\cf(\k) = \w$, then $\cofinal = \k^+$.
\end{itemize}
The reason for calling this principle $\schh$ is that the assertion ``\schh holds for all $\k > \continuum$'' is equivalent to the Singular Cardinals Hypothesis, abbreviated $\mathsf{SCH}$.
This is because if $\k > \continuum$, then $\cofinal = \cov(\k^\w) = |\k^\w|$; so \schh holds for all $\k > \continuum$ if and only if $\k^{\aleph_0} = \k^+$ for all $\k > \continuum$ with $\cf(\k) = \w$.
(And this is equivalent to $\mathsf{SCH}$ by a theorem of Silver.) 

In \cite[Section 2]{Fuchino&Soukup}, Fuchino and L. Soukup introduce
a weak version of Jensen's principle $\square_\k$ denoted $\square^{***}_{\w_1,\k}$.
As one might guess from the notation, Fuchino and Soukup discuss a more generalized $2$-parameter version of this principle, $\square^{***}_{\lambda,\k}$. We will only need the special case $\lambda = \w_1$, or rather a consequence of it described below.

Recall that $\square^{*}_{\k}$ denotes the \emph{weak square} principle at $\k$. This is a weakening of $\square_\k$, and is equivalent to the existence of a special $\k^+$-Aronszajn tree \cite{Jensen}.

\begin{lemma}\label{lem:square}
\emph{(Fuchino and Soukup, \cite[Lemma 4]{Fuchino&Soukup})}
Assume \schh, and let $\k$ be a singular cardinal with $\cf(\k) = \w$. Then $\square^{*}_\k$ implies $\square^{***}_{\w_1,\k}$.
\end{lemma}

It is also worth mentioning that Foreman and Magidor in \cite{Foreman&Magidor} considered an axiom called the \emph{very weak square} principle for $\k$. The very weak square principle for $\k$ implies $\square^{***}_{\w_1,\k}$ unconditionally. 

Instead of applying $\square^{***}_{\w_1,\k}$ directly, we use a consequence of $\schh + \square^{***}_{\w_1,\k}$ proved by Fuchino and Soukup in \cite{Fuchino&Soukup}, namely the existence of something called a $(\w_1,\k)$-Jensen matrix.

Let $p$ be a set and $\k$ a cardinal, and let $\theta$ be a regular cardinal with $\w_1,2^\k,|\mathrm{tc}(p)| < \theta$, where $\mathrm{tc}(p)$ denotes the transitive closure of $p$. Let $H_\theta$ denote the set of all sets hereditarily smaller than $\theta$, and recall $H_\theta \models \zfc^-$.
A \emph{$(\w_1,\k)$-Jensen matrix} over $p$ is a matrix $\seq{M_{\a,n}}{\a < \k^+, n < \w}$ of elementary submodels of $H_\theta$ such that
\begin{enumerate}
\item For each $\a < \k^+$ and $n < \w$, $p \in M_{\a,n}$, $\k+1 \sub M_{\a,n}$, and $\card{M_{\a,n}} < \k$.
\item For each $\a < \k^+$, $\seq{M_{\a,n}}{n < \w}$ is an increasing sequence.
\item If $\a < \k^+$ and $\cf(\a) > \w$, then there is some $n^* < \w$ such that, for every $n \geq n^*$, $[M_{\a,n}]^{\aleph_0} \cap M_{\a,n}$ is cofinal in $\hspace{-.1mm} \big( [M_{\a,n}]^{\aleph_0},\sub \hspace{-1mm} \big)$.
\item For each $\a < \k^+$, let $M_\a = \bigcup_{n < \w}M_{\a,n}$. Then $\seq{M_\a}{\a < \k^+}$ is continuously increasing, and $\k^+ \sub \bigcup_{\a < \k^+}M_\a$.
\end{enumerate}

\noindent The following is proved as Theorems 7 and 8 in \cite{Fuchino&Soukup}:

\begin{lemma}\emph{(Fuchino and Soukup, \cite[Theorems 7 and 8]{Fuchino&Soukup})}
Assume \schh,
and let $\k$ be a singular cardinal with $\cf(\k) = \w$.
Then $\square^{***}_{\w_1,\k}$ holds if and only if
there is a $(\w_1,\k)$-Jensen matrix over $p$ for some (or any) $p$.
\end{lemma}

We are now ready to prove the main technical result of this section. Compare the following with Theorems~\ref{thm:induct} and \ref{thm:dichotomy}.

\begin{theorem}\label{thm:Davies}
Assume \schh and $\square^{***}_{\w_1,\k}$,
and let $\k$ be a singular cardinal with $\cf(\k) = \w$.
If $X$ is any completely metrizable space of weight $\k$, there is a partition of $X$ into $\leq\!\k^+$ completely metrizable spaces of weight $<\!\k$.
\end{theorem}
\begin{proof}
Let $X$ be a completely metrizable space of weight $\k$. Applying Lemma~\ref{lem:metricbasics}, fix a basis $\B$ for $X$ such that $|\B| = \k$ and every point of $X$ is contained in only countably many members of $\B$. Applying \schh, let $\D$ be a cofinal family in $\hspace{-.1mm} \big( [\B]^{\aleph_0},\sub \hspace{-1mm} \big)$ with $|\D| = \k^+$, and fix a bijection $p: \k^+ \to \D$.
Let $\seq{M_{\a,n}}{\a < \k^+, n < \w}$ be a $(\w_1,\k)$-Jensen matrix over $p$.
As in item $(4)$ above, let $M_\a = \bigcup_{n \in \w}M_{\a,n}$ for every $\a < \k^+$.

For each $\a < \k^+$ and $n < \w$, let
$$X_{\a,n} = \set{x \in X}{\set{U \in \B}{x \in U} \sub M_{\a,n}}$$
and let
$$\textstyle Y_{\a,n} = X_{\a,n} \setminus \bigcup \set{X_{\b,m}}{\text{either }\b < \a, \text{ or else }\b = \a \text{ and }m < n}.$$
We claim that $\set{Y_{\a,n}}{\a < \k^+, n < \w}$ is a partition of $X$ into $\leq\!\k^+$ completely metrizable spaces, each of weight $<\!\k$.

It is clear from the definitions that $\set{Y_{\a,n}}{\a < \k^+,n < \w}$ is a partition of $\bigcup \set{X_{\a,n}}{\a < \k^+,n < \w}$, and that each $Y_{\a,n}$ has $\set{U \cap Y_{\a,n}}{U \in \B \cap M_{\a,n}}$ as a basis, and therefore $\wt(Y_{\a,n}) \leq |M_{\a,n}| < \k$.
To finish the proof of the theorem, it remains to show that $\bigcup \set{X_{\a,n}}{\a < \k^+,n < \w} = X$ and that each $Y_{\a,n}$ is completely metrizable.

Let $x \in X$. Then $\set{U \in \B}{x \in U}$ is countable (by our choice of $\B$), and there is some $A \in \D$ with $\set{U \in \B}{x \in U} \sub A$ (by our choice of $\D$). There is some $\g < \k$ with $p(\g) = A$, and by $(4)$ there is some $\a < \k^+$ and $n < \w$ with $\g \in M_{\a,n}$. As $p \in M_{\a,n}$ by $(1)$, we have $p(\g) = A \in M_{\a,n}$ and therefore $A \sub M_{\a,n}$. (The ``therefore'' part uses the elementarity of $M_{\a,n}$ in $H_\theta$, which implies countable members of $M_{\a,n}$ are subsets of $M_{\a,n}$.) Hence $x \in X_{\a,n}$. As $x$ was arbitrary, it follows that $\bigcup \set{X_{\a,n}}{\a < \k^+,n < \w} = X$.

It remains to show each $Y_{\a,n}$ is completely metrizable.
For each $\a < \k^+$, define $X_\a = \bigcup_{n < \w} X_{\a,n}$.
Observe that each $X_{\a,n}$ is closed in $X$, and each $X_\a$ is therefore an $F_\s$ subset of $X$.

Fix $\a < \k^+$ and $n < \w$. We consider three cases.

First, suppose $\a = \g+1$ is a successor ordinal. By $(4)$, 
$$\textstyle \bigcup \set{M_{\b,m}}{\b < \a, m < \w} \,=\, \bigcup_{\b \leq \g}M_\b \,=\,  M_\g \,=\, \bigcup \set{M_{\g,m}}{m < \w}.$$
From this and our definitions, it follows that
$$\textstyle \bigcup \set{X_{\b,m}}{\b < \a, m < \w} \,=\, \bigcup \set{X_{\g,m}}{m < \w} = X_\g.$$
Consequently,
\begin{align*}
Y_{\a,n} &\,=\, \textstyle X_{\a,n} \setminus \bigcup \set{X_{\b,m}}{\text{either }\b < \a, \text{ or else }\b = \a \text{ and }m < n} \\
&\,=\, \textstyle X_{\a,n} \setminus \big( X_\g \cup \bigcup \set{X_{\a,m}}{m < n} \big).
\end{align*}
Because $X_\g$ is an $F_\s$ and each of the $X_{\a,m}$ is closed, this shows that $Y_{\a,n}$ is a $G_\dlt$ subset of $X$.

Next, suppose $\a$ is a limit ordinal with $\cf(\a) = \w$. Fix an increasing sequence $\seq{\g_k}{k < \w}$ of ordinals with limit $\a$. Arguing in a way similar to the previous paragraph, $(4)$ implies
$$\textstyle \bigcup \set{M_{\b,m}}{\b < \a, m < \w} \,=\, \textstyle \bigcup_{k \in \w} M_{\g_k} \,=\, \bigcup \set{M_{\g_k,m}}{k,m < \w}.$$
From this and the definition of the $X_{\b,m}$, it follows that
$$\textstyle \bigcup \set{X_{\b,m}}{\b < \a, m < \w} \,=\, \bigcup \set{X_{\g_k,m}}{k,m < \w} = \bigcup_{k < \w}X_{\g_k}.$$
Consequently, just as in the previous case,
\begin{align*}
Y_{\a,n} &\,=\, \textstyle X_{\a,n} \setminus \bigcup \set{X_{\b,m}}{\text{either }\b < \a, \text{ or else }\b = \a \text{ and }m < n} \\
&\,=\, \textstyle X_{\a,n} \setminus \big( \bigcup_{k < \w}X_{\g_k} \cup \bigcup \set{X_{\a,m}}{m < n} \big),
\end{align*}
and this shows that $X_{\a,n}$ is a $G_\dlt$ subset of $X$.

Finally, suppose $\a$ is a limit ordinal with $\cf(\a) > \w$.
By $(3)$, there is some $n^* < \w$ such that, for every $m \geq n^*$, $[M_{\a,m}]^{\aleph_0} \cap M_{\a,m}$ is cofinal in $\hspace{-.1mm} \big( [M_{\a,m}]^{\aleph_0},\sub \hspace{-1mm} \big)$.
Let $k = \max\{n,n^*\}$.
Fix $x \in X_{\a,n}$.
By $(2)$, $M_{\a,n} \sub M_{\a,k}$ and therefore $x \in X_{\a,k}$; by definition, this means $\set{U \in \B}{x \in U} \sub M_{\a,n}$.
By our choice of $\B$, $\set{U \in \B}{x \in U}$ is countable.
Because $k \geq n^*$, there is a countable $A \in M_{a,n}$ such that $\set{U \in \B}{x \in U} \sub A$.
Now $A \in M_{\a,n}$ implies $A \in M_\a$, and by $(4)$ (specifically, the ``continuously increasing'' part), there is some $\b < \a$ such that $A \in M_\b$. In particular, this means $A \in M_{\b,m}$ for some $m < \w$.
By the elementarity of $M_{\b,m}$ in $H_\theta$ and the countability of $A$, $A \in M_{\b,m}$ implies $\set{U \in \B}{x \in U} \sub A \sub M_{\b,m}$, which implies $x \in X_{\b,m}$. 
Hence $x \notin Y_{\a,n}$.
Because $x$ was an arbitrary member of $X_{\a,n}$, this shows $Y_{\a,n} = \0$.
In particular, $Y_{\a,n}$ is Polish.
\end{proof}

\begin{theorem}\label{thm:main}
Suppose \schh holds, and $\square^{***}_{\w_1,\mu}$ holds for every singular $\mu$ with $\cf(\mu) = \w$. Then for any uncountable cardinal $\k$,
\begin{itemize}
\item[$\circ$] If $\cf(\k) > \w$, then $\cov(X) = \mp(X) = \k$.

\vspace{1.5mm}

\item[$\circ$] If $\cf(\k) = \w$, then either

\vspace{1mm}

\begin{itemize}
\item[\tiny{$\circ$}] every closed subspace of $X$ is locally $<\!\k$-like, in which case $\cov(X) = \mp(X) = \k$, or

\vspace{1mm}

\item[\tiny{$\circ$}] some closed subspace of $X$ is not locally $<\!\k$-like, in which case $\cov(X) = \mp(X) = \k^+$.
\end{itemize}
\end{itemize}
In particular, $\cov(X) = \mp(X)$ for every completely metrizable space $X$.
\end{theorem}
\begin{proof}
Recall that Theorem~\ref{thm:bounds} states $\k \leq \cov(X) \leq \mp(X) \leq \cofinal$ whenever $X$ is a completely metrizable space of weight $\k$. Therefore \schh implies $\cov(X) = \mp(X) = \k$ whenever $\wt(X) = \k$ and $\cf(\k) > \w$. 
It remains to prove the $\cf(\k) = \w$ case of the theorem.

For this case, note that Corollary~\ref{cor:upperbound} and \schh combine to give us the lower bounds we need. Specifically: if $\k$ is an uncountable cardinal with $\cf(\k) = \w$, then for every completely metrizable space $X$ of weight $\k$,
\begin{itemize}
\item[$\circ$] $\k \leq \cov(X) \leq \mp(X)$.

\vspace{1mm}

\item[$\circ$] If some closed subspace of $X$ is not locally $<\!\k$-like, then $\k^+ = \cofinal = \cov(X) \leq \mp(X)$.
\end{itemize}

\noindent To finish the proof, it suffices to show that for every uncountable $\k$ with $\cf(\k) = \w$, the following statement, which we denote by $(\mathrm{IH}_\k)$, holds.

\vspace{1mm}

\noindent $(\mathrm{IH}_\k)\!:\,$ For every completely metrizable space $X$ of weight $\k$,
\begin{itemize}
\item[$\circ$] If every closed subspace of $X$ is locally $<\!\k$-like, then there is a partition of $X$ into $\leq\!\k$ Polish spaces.

\vspace{1mm}

\item[$\circ$] If some closed subspace of $X$ is not locally $<\!\k$-like, then there is a partition of $X$ into $\leq\!\k^+$ Polish spaces.
\end{itemize}

\noindent For convenience, if $\cf(\k) > \w$ then we let $(\mathrm{IH}_\k)$ denote the assertion that $\cov(X) = \mp(X) = \k$ whenever $X$ has weight $\k$. As noted above, we know already that $(\mathrm{IH}_\k)$ is true for all $\k$ with $\cf(\k) > \w$.

We prove $(\mathrm{IH}_\k)$ holds for all uncountable $\k$ by transfinite induction on $\k$.
The only yet unproven steps in the induction are limit steps of countable cofinality.
So fix $\k > \aleph_0$ with $\cf(\k) = \w$, and suppose $(\mathrm{IH}_\mu)$ holds for all uncountable $\mu < \k$. Let $X$ be a completely metrizable space with $\wt(X) = \k$. There are two cases.

First, suppose every closed subset of $X$ is locally $<\!\k$-like.
By Theorem~\ref{thm:dichotomy}, there is a partition $\P$ of $X$ into $\leq\!\k$ completely metrizable spaces each of weight $<\!\k$.
Because $(\mathrm{IH}_\mu)$ holds for all uncountable $\mu < \k$, for each $Y \in \P$ there is a partition $\P_Y$ of $Y$ into $\leq\!\k$ Polish spaces. Then $\bigcup \set{\P_Y}{Y \in \P}$ is a partition of $X$ into $\leq\!\k$ Polish spaces.

Next, suppose some closed subspace of $X$ is not locally $<\!\k$-like. By Theorem~\ref{thm:Davies}, there is a partition $\P$ of $X$ into $\leq\!\k^+$ completely metrizable spaces, each of weight $<\!\k$.
Because $(\mathrm{IH}_\mu)$ holds for all uncountable $\mu < \k$, for each $Y \in \P$ there is a partition $\P_Y$ of $Y$ into $\leq\!\k$ Polish spaces. Then $\bigcup \set{\P_Y}{Y \in \P}$ is a partition of $X$ into $\leq\!\k^+$ Polish spaces.
\end{proof}

As noted in the introduction, this theorem is only really interesting in the case $\k \leq \continuum$.
However, this restriction is not required at any point in the proof and is omitted intentionally from the statement of Theorem~\ref{thm:main}.
It is worth pointing out that, with a little more work, one need only assume $\square^{***}_\mu$ when $\cf(\mu) = \w$ and $\aleph_\w \leq \mu < \continuum$. The reason is that if $\k > \continuum$ then $\cov(X) = \mp(X) = |X|$: then \schh, Theorem~\ref{thm:dichotomy}, and \cite[Theorem 4.1.15]{Engelking} allow us to compute that $|X| = \k$, unless $\cf(\k) = \w$ and some closed subspace of $X$ is not locally $<\!\k$-like, in which case $|X| = \k^{\aleph_0} = \k^+$.

\begin{theorem}
Suppose there is some completely metrizable space $X$ with $\cov(X) < \mp(X)$. Then $0^\dagger$ exists.
\end{theorem}
\begin{proof}
If $\square^{***}_\k$ fails for a singular cardinal $\k$, then $\square_\k$ fails also by Lemma~\ref{lem:square}.
The failure of $\square_\k$ for a singular cardinal $\k$ implies the existence of $0^\dagger$, and in fact much more, by results of Cummings and Friedman \cite{CF}.

Therefore, by the previous theorem, it suffices to show that the failure of \schh fails implies $0^\dagger$ exists.

By an \emph{inner model} we mean a (definable, with parameters) class $M$ such that $M \models \zfc$. An \emph{inner model of \gch} means additionally $M \models \gch$.
We say the Covering Lemma holds over an inner model $M$ if for every uncountable $X \sub M$, there is some $Y \in M$ with $Y \supseteq X$ such that $|X| = |Y|$.

It is proved in \cite{Just} (cf. Lemmas 4.9 and 4.10) that if the Covering Lemma holds over an inner model of \gch, then \schh holds.

By work of Dodd and Jensen \cite{DJ}, if there is no inner model containing a measurable cardinal, then the Covering Lemma holds over an inner model of \gch (specifically, the Dodd-Jensen core model $K^{\mathrm{DJ}}$). By further work of Dodd and Jensen \cite{DJ2}, if there is an inner model containing a measurable cardinal but $0^\dagger$ does not exist, then the Covering Lemma holds over an inner model of \gch (specifically, either $L[U]$ for some measure $U$ with $crit(U)$ as small as possible, or $L[U,C]$ for some $C$ Prikry-generic over $L[U]$).

Either way, if $0^\dagger$ does not exist then Jensen's Covering Lemma holds over an inner model $M$ of \gch, and \schh follows by \cite{Just}.
%
%
\end{proof}

This argument leaves open the precise consistency strength of $\neg\schh$.

\begin{question}
Is $\neg\schh$ equiconsistent with the existence of a measurable cardinal $\mu$ of Mitchell order $\mu^{++}$?
\end{question}

\noindent By work of Gitik \cite{Gitik1,Gitik2}, the failure of $\mathsf{SCH}$ is equiconsistent with the existence of a measurable cardinal $\mu$ of Mitchell order $\mu^{++}$. 
An affirmative answer to this question would raise the consistency strength of the strict inequality $\cov(X) < \mp(X)$ to measurable cardinals of high Mitchell order (cf. \cite{CF}).
On the other side of things, it would also be interesting to reduce the upper bound from Section~\ref{sec:<} on the consistency strength of this statement.

\begin{question}
Beginning with a supercompact cardinal, can one force $\cov(X) < \mp(X)$ for some completely metrizable space $X$?
\end{question}

\vspace{2mm}

\begin{center} 
\textsc{Acknowledgments} 
\end{center}

\vspace{1.5mm}

I would like to thank Alan Dow for helpful discussions about some of the ideas in this paper, and Alexander Osipov for raising some questions related to the material in Section 2. I would also like to thank Yair Hayut for his observations in \cite{Hayut}. An early version of the proof of Theorem~\ref{thm:Davies} used a much stronger $L$-like hypothesis, and Hayut proved the negation of this hypothesis equiconsistent with the existence of a single inaccessible. This inspired a search for a more efficient proof, eventually raising the lower bound on the consistency strength of $\cov(X) < \mp(X)$ to $0^\dagger$.



\end{document}